\documentclass[11pt]{article}
\usepackage{amsmath}
\usepackage{mathtools}
\usepackage{amsthm}
\usepackage{amssymb}
\usepackage{amsfonts}

\numberwithin{equation}{section}
\usepackage{indentfirst}
\usepackage{latexsym}
\usepackage{graphicx}
\usepackage{flafter}
\usepackage{caption}
\usepackage{textcomp}
\usepackage{supertabular}
\usepackage{longtable, booktabs}
\usepackage{array}
\usepackage{hhline}
\usepackage{bigstrut,bigdelim}
\usepackage{rotating}
\usepackage{multicol}
\usepackage{multirow}
\usepackage{makeidx}
\usepackage{fancyhdr}
\usepackage{threeparttable}
\usepackage{mathrsfs}
\usepackage[all]{xy}
\usepackage{xcolor}
\usepackage[numbers,sort&compress]{natbib}
\usepackage{hyperref}
\theoremstyle{definition}
\newif\ifger
\gerfalse
\newtheorem{definition}{Definition}[section]
\newtheorem{theorem}{Theorem}[section]
\newtheorem{lemma}[definition]{Lemma}

\newtheorem{hypothesis}{Hypothesis}
\newtheorem{conjecture}{Conjecture}[section]
\newtheorem{proposition}[definition]{Proposition}

\newtheorem*{availability of data and material}{Availability of Data and Material}
\newtheorem*{conflict of interest}{Conflict of interest}

\newcommand{\res}[2]{\left\langle#1\right\rangle_{#2}}

\begin{document}
\baselineskip=19pt

\title{Permutation binomials of the form \\[2mm]
$X^r(X^{q-1}+a)$ over finite fields}

\author{
Junna Ni$^a$, Xuan Pang$^{a,*}$, Jianhua Yu$^a$, Pingzhi Yuan$^a$\\
\small\itshape
$^a$School of Mathematical Sciences, South China Normal University\\
\small
E-mail addresses:
\texttt{nijunna@126.com} (J. Ni),
\texttt{pangxuan202503@163.com} (X. Pang),\\
\small
\texttt{yujianhuascnu@126.com } (J. Yu),
\texttt{yuanpz@scnu.edu.cn} (P. Yuan)\\
\small
$^*$Corresponding author
}

\date{}
\maketitle

\begin{abstract}
This paper is devoted to studying permutation binomial
$F_{r,a}(X)=X^r(X^{q-1}+a)\in\mathbb F_{q^e}[X]$ with $a\in\mathbb F_{q^e}^*$. We present a complete characterization for $F_{r,a}$ to be a permutation of $\mathbb{F}_{q^e}$. This yields a complete proof of the conjecture proposed by Masuda--Rubio--Santiago \cite{masuda2022permutation}. More precisely, we show that \(F_{r,a}\) permutes \(\mathbb F_{q^e}\) if and only if $(-a)^{(q^e-1)/(q-1)}\ne1, \gcd(r,q-1)=1$,
and there exists \(1\le h<e\) with \(\gcd(h,e)=1\) such that $r(q^h-1)\equiv q-1\pmod{q^e-1}.$



\end{abstract}

\section{Introduction}

Let \(\mathbb F_q\) be the finite field of \(q\) elements, where \(q\) is a prime power. A polynomial \(f(X)\in\mathbb F_q[X]\) is called a permutation polynomial if it induces a permutation of \(\mathbb F_q\).
Permutation polynomials are fundamental objects in finite field theory and have been extensively studied because of both their algebraic interest and their applications in areas such as coding theory \cite{Ding, Ding-Zhou,Laigle-Chapuy},
cryptography \cite{Rivest-Shamir-Adelman, Schwenk-Huber, Singh2009, Singh2020} and combinatorial designs \cite{Ding-Yuan}. In particular, determining the
permutation behavior of polynomials with few terms remains an
interesting and challenging problem.

Among sparse permutation polynomials, permutation binomials have been studied by means of
character sums~\cite{Masuda2009}, the multivariate method~\cite{Dobbertin1999almost,WangYP2018}, the fractional approach~\cite{AGW2011,Wan1991,Wang2007,zieve2007}, and related approaches \cite{Ayad2015, Hou2016,Tu2021,Hou2013,Hou_st2015, Wan1994,Hou2015survey,PB2024survey}. In fact, after a suitable power substitution, a general binomial \(X^m+aX^n\) can be reduced to the form $X^r\bigl(X^{(q-1)/d}+a\bigr)$, where $d\mid q-1$.
 Nevertheless, a general characterization remains difficult. Even for the structured family $X^r\bigl(a+X^{t(q-1)}\bigr)$, necessary and sufficient conditions are known only in special cases. For more information related to permutation binomials, the reader is referred to the surveys~\cite{Hou2015survey,PB2024survey} and the references therein.

In this paper, we study permutation binomials over \(\mathbb F_{q^e}\) of the form $F_{r,a}(X)=X^r(X^{q-1}+a)$, where \(e\ge2\), \(r\ge1\), and
\(a\in\mathbb F_{q^e}^*\). Masuda, Rubio and Santiago~\cite{masuda2022permutation} investigated
this family and obtained a complete classification for
\(e\in\{2,3,4\}\) over arbitrary prime powers \(q\), and for
\(e\in\{5,6\}\) when \(q\) is an odd prime. For general \(e\), they also characterized such permutation binomials that can be expressed as the composition of a monomial with a linearized binomial. Moreover, based on a computer search for \(q^e<10^8\), they further proposed the following conjecture.

\begin{conjecture}\cite[Conjecture 1.5]{masuda2022permutation}\label{conj}
Let $F_{r,a}(x) = X^r(X^{q-1}+a) \in \mathbb{F}_{q^e}^*[X]$ with $e \ge 2$, and let $L =(q^e-1)/(q-1)$.  Then $F_{r,a}(X)$ permutes $\mathbb{F}_{q^e}$ if and only if $F_{r,a}(X)\equiv \left(X^{q^h} + aX\right)\circ X^r \pmod{X^{q^e} - X}$, where $(-a)^L \neq 1$ and $\gcd(r,q-1)=1$.
\end{conjecture}

The purpose of the present paper is to prove this conjecture.
Our main theorem gives the following complete characterization.

\begin{theorem}\label{thm:main}
Let $F_{r,a}(X)=X^r(X^{q-1}+a)\in\mathbb{F}_{q^e}^*[x]$ with $r\ge1$ and $e\geq 2$.  Then $F_{r,a}(X)$  permutes $\mathbb{F}_{q^e}$ if and only if the following conditions hold: 

(i) \((-a)^{(q^e-1)/(q-1)}\ne1\);

(ii) \(\gcd(r,q-1)=1\);

(iii) there exists \(1\le h<e\) with \(\gcd(h,e)=1\) such that $r(q^h-1)\equiv q-1\pmod{q^e-1}.$ Equivalently,  \(rL_h\equiv1\pmod L\), where $L=(q^e-1)/(q-1)$ and $L_h=(q^h-1)/(q-1)$.
\end{theorem}

The first two necessary conditions are elementary and can be addressed immediately.

\begin{proposition}\label{conditions_for_a,q}
If $F_{r,a}(X)$ permutes $\mathbb F_{q^e}$, then $(-a)^{L}\neq 1$ and $\gcd(r,q-1)=1$. Moreover, for \(q=2\), there is no such polynomial that permutes $\mathbb F_{q^e}$.
\end{proposition}
\begin{proof}
Suppose that $F_{r,a}$ permutes $\mathbb{F}_{q^e}$. Since $\mathbb{F}_{q^e}^*$ is cyclic of order $q^e-1$, the equation $x^{q-1}=-a$ has a nonzero solution in $\mathbb{F}_{q^e}^*$ if and only if $(-a)^{(q^e-1)/(q-1)}=1$. Such a nonzero root would yield two distinct zeros for $F_{r,a}$, contradicting the permutation property. Hence $(-a)^{(q^e-1)/(q-1)}\neq1$.

If $\gcd(r,q-1)>1$, there exists $\lambda\in\mathbb{F}_q^*$ with $\lambda\ne1$ and $\lambda^r=1$. Since $\lambda^{q-1}=1$, we have $F_{r,a}(\lambda x)=\lambda^rF_{r,a}(x)=F_{r,a}(x)$, again contradicting injectivity. When $q=2$, the distinct elements $0$ and $a$ are both zeros of $F_{r,a}$.
\end{proof}

The main problem is therefore the exponent condition.  For \(1\le h<e\), the proof will write \(\bar rL_h=n_hL+\rho_h\) and show that \(q\mid n_h\) for every \(h\), where $L=(q^e-1)/(q-1)$ and \(\bar r=\res{r}{L}\) denotes the least nonnegative residue of
\(r\) modulo \(L\).  The quotient divisibility is not the final statement of the conjecture; it is an intermediate arithmetic form from which the exponent congruence will be reconstructed in the last section.

The remainder of this paper is organized as follows. In Section~2,
we recall the basic facts on Galois rings and \(p\)-adic valuations
needed in the proof. In Section~3, we lift the Hermite power-sum
identity to Galois rings and derive the initial restrictions on the
exponent. Section~4 proves the divisibility of the quotient parameters
\(n_h\). In Section~5, we use a cyclic-shift argument to obtain the
exponent condition. Finally, Section~6 completes the proof of
Theorem~\ref{thm:main}.


\section{Preliminaries}
In this section, we collect the basic facts on Galois rings and $p$-adic valuations of binomial coefficients that will be used in the sequel (see \cite[Section 6]{bini2002finite} and \cite{Kummer1852}).

Throughout this paper, let $q=p^m$ be a prime power and $e\ge 2$. For $0\le h\le e$, define $L_h=(q^h-1)/(q-1)=1+q+\cdots+q^{h-1}$ with $L_0=0$, and let $L=L_e$. Then $(q-1)L_h=q^h-1$ and $L_h=qL_{h-1}+1$ for all $1\le h\le e$, and $L=q^hL_{e-h}+L_h$ for $0\le h\le e$. Given an integer~$z$ and positive integer $M$, let $\langle z\rangle_M$ denote the least nonnegative residue of $z$ modulo $M$. For a nonzero integer $z$, let $v_p(z)=v_p(|z|)$, the largest nonnegative integer $j$ such that $p^j\mid z$.

\subsection{Galois rings and Teichm\"uller representatives}

Let $p$ be a prime and $t,n$ be positive integers. A Galois ring $R=GR(p^t,n)$ is a finite local ring of
characteristic $p^t$ and order $p^{tn}$, with unique maximal ideal $pR$ and residue field  $R/pR\cong\mathbb F_{p^n}$.

Equivalently, if $\phi(U)\in\mathbb Z[U]$ is a monic polynomial of degree $n$ whose reduction modulo~$p$ is irreducible, then $R\cong\mathbb{Z}_{p^t}[U]/(\phi(U))$, where \(\mathbb Z_{p^t}=\mathbb Z/p^t\mathbb Z\).
In particular, the natural reduction map
$\pi:R\longrightarrow R/pR\cong\mathbb F_{p^n}$
is surjective, and $z\in R$ is a unit if and only if  $\pi(z)$ is nonzero.  Moreover, $\mathbb{Z}_{p^t}$ embeds naturally into $R$, and thus an integer vanishes in \(R\) precisely when it is divisible by~\(p^t\).

It is well known that each $z\in R$ has a unique $p$-adic expansion $z=\sum_{i=0}^{t-1}p^iz_{i}$, where each $z_i$ belongs to the set $\mathcal{T}_n:=\{0,1,\xi,\dots,\xi^{p^n-2}\}$ called the {\it Teichm\"uller set} of $R$. Moreover, $\mathcal{T}_n^*=\langle\xi\rangle$ is a cyclic subgroup of $R^*$ of order $p^n-1$. Thus $\mathcal{T}_n^*\cong\mathbb{F}_{p^n}^*$.


\subsection{A $p$-adic valuation formula for binomial coefficients}
By Kummer's theorem, for a prime $p$ and integers $0\le M\le N$, the $p$-adic valuation $v_p\binom{N}{M}$ equals the number of
carries occurring when $M$ and $N-M$ are added in base $p$. The following consequence gives the particular valuation formula
needed later.

\begin{lemma}\label{lem:kummer_special}
Let $\alpha,\beta$ be positive integers, and suppose that $N=p^\alpha(p^\beta-1)$, $M=p^\alpha Q+E$,
where $Q$ is an integer, $0\leq E<p^\alpha$, and $0\le M\le N$. Then
\begin{equation*}
v_p\binom{N}{M}=
\begin{cases}
0, & E=0,\\[1mm]
\alpha+v_p(Q+1)-v_p(E), & E>0.
\end{cases}
\end{equation*}
\end{lemma}
\begin{proof}
If $E=0$, then $M=p^\alpha Q$. By Kummer's theorem,
$v_p\binom{N}{M}
=v_p\binom{p^\beta-1}{Q}=0,$
since every base-$p$ digit of $p^\beta-1$ is $p-1$. Now suppose $E>0$. Since $0<E<p^\alpha$,
$v_p(M)=v_p(E)$. Using
$\binom{N}{M}=
\frac{N}{M}\binom{N-1}{M-1},$
we obtain
\begin{equation}\label{eq:kummer_special}
v_p\binom{N}{M}=
\alpha-v_p(E)+v_p\binom{N-1}{M-1}.\end{equation}
Now write
$N-1=(p^\beta-2)p^\alpha+(p^\alpha-1)$,
and $M-1=Qp^\alpha+(E-1)$.
The lowest $\alpha$ base-$p$ digits contribute no carries, so Kummer's theorem gives
$v_p\binom{N-1}{M-1}=v_p\binom{p^\beta-2}{Q}$.
Since $\binom{p^\beta-1}{Q+1}=
\frac{p^\beta-1}{Q+1}
\binom{p^\beta-2}{Q}$
and $v_p\binom{p^\beta-1}{Q+1}=0$,
it follows that $v_p\binom{p^\beta-2}{Q}=v_p(Q+1)$. Substituting into Equation \eqref{eq:kummer_special} yields the desired conclusion.
\end{proof}

\section{Lifted Hermite power sums and preliminary restrictions}
Let $q=p^m$ and $n=me$. Recall that if $f(X)\in\mathbb F_{q^e}[X]$ induces a permutation of $\mathbb F_{q^e}$, then, for every $1\le S<q^e-1$,
\begin{equation}\label{Hermit}
\sum_{X\in\mathbb F_{q^e}}f(X)^S
=\sum_{X\in\mathbb F_{q^e}}X^S=0.
\end{equation}
This is the power-sum form of the Hermite criterion. However,
when direct application in finite fields, it lacks sufficient information for our problem. Indeed, in the expansion of $\left(X^r(X^{q-1}+a)\right)^S$, all integer coefficients divisible by $p$ vanish in $\mathbb{F}_{q^e}$, erasing information about higher $p$-adic divisibility. To preserve this information, we lift the power-sum identities to a suitable Galois ring.


Choose a monic irreducible
polynomial $\overline{\varphi}(U)\in\mathbb F_p[U]$ of degree~$n$ such that $\mathbb F_{q^e}\cong
\mathbb F_p[U]/(\overline{\varphi}(U))$, and let $\varphi(U)\in\mathbb Z[U]$ be a monic lift of
$\overline{\varphi}(U)$. For $t\ge1$, define $\mathcal A_t=\mathbb Z_{p^t}[U]/(\varphi(U))$.
Then $\mathcal A_t\cong \operatorname{GR}(p^t,n)$, with residue field $\mathbb F_{q^e}$. Denote $\mathcal{T}_t$ as the Teichm\"uller set of $\mathcal A_t$.

By Section~2.1, an integer $z\in\mathcal{A}_t$ is zero if and only if $p^t\mid z$. Let $\pi_t:\mathcal A_t\longrightarrow\mathbb F_{q^e}$ be the natural reduction map. Its restriction $\pi_t|_{\mathcal T_t}:\mathcal T_t\longrightarrow\mathbb F_{q^e}$
is a bijection, and $\pi_t|_{\mathcal T_t^*}:
\mathcal T_t^*\longrightarrow\mathbb F_{q^e}^*$
is a group isomorphism. We denote its inverse bijection
$\tau_t:\mathbb F_{q^e}\longrightarrow\mathcal T_t$,
as the Teichm\"uller lift.

This isomorphism also yields the following power-sum formula over $\mathcal T_t$. Since $\mathcal T_t^*$ is cyclic of order $q^e-1$, the finite geometric-series identity gives, for every positive integer $M$,
\begin{equation}\label{eq:power_sumTt}
\sum_{x\in \mathcal T_t}x^M=
\begin{cases}
q^e-1, & q^e-1\mid M,\\
0, & q^e-1\nmid M.
\end{cases}
\end{equation}

We now provide a lifted analogue of \eqref{Hermit} over $\mathcal{A}_t$.


\begin{lemma}\label{lem:lifted_power_sum}
Let $1\le S<q^e-1$ and  $t=1+v_p(S)$. If
$G\in \mathcal{A}_t[X]$ reduces modulo $p$ to be a permutation of $\mathbb F_{q^e}$, then $\sum_{x\in \mathcal T_t}G(x)^S=0$ in $\mathcal{A}_t$.
\end{lemma}

\begin{proof}
Let $\overline G$ denote the reduction of $G$ modulo $p$.
For every $x\in\mathcal T_t$, $\pi_t(G(x))=\overline G(\pi_t(x))$. Hence, if $y_x=\tau_t\bigl(\overline G(\pi_t(x))\bigr)\in\mathcal T_t$,
then $G(x)\equiv y_x\pmod p$. Thus we may write
$G(x)=y_x+pz_x$ for some $z_x\in\mathcal A_t$.
We claim that $G(x)^S=y_x^S$ in $\mathcal A_t$. Indeed,
$$G(x)^S-y_x^S=\sum_{i=1}^{S}
\binom{S}{i}y_x^{S-i}p^iz_x^i.$$
Using
$i\binom{S}{i}=S\binom{S-1}{i-1},$
we have
$v_p\binom{S}{i}\ge v_p(S)-v_p(i).$
Therefore
$$v_p\!\left(p^i\binom{S}{i}\right)
\ge v_p(S)+i-v_p(i)
\ge v_p(S)+1=t,$$
where we use $v_p(i)\le i-1$. This implies that $p^t\mid p^i\binom{S}{i}$ for $1\le i\le S$. Hence
$G(x)^S=y_x^S$. The permutation assumption makes $y_x$ run through $\mathcal{T}_t$, so \eqref{eq:power_sumTt} yields the desired result.
\end{proof}

\begin{proposition}\label{prop:lifted-binomial}
Assume that $F_{r,a}$ induces a permutation of $\mathbb{F}_{q^e}$.  Let $k$ be a positive integer such that  $S=(q-1)k<q^e-1$.  Then, for any lift $\widehat a\in \mathcal{A}_{1+v_p(S)}$ of $a$,
\begin{equation}\label{eq:lifted-sum}
\sum_{\substack{0\le J\le S\\ rk+J\equiv0\pmod L}}
\binom SJ\widehat a^{\,S-J}=0
\qquad\text{in }\mathcal{A}_{1+v_p(S)}.
\end{equation}
An integer \(J\) is called {\it admissible} for \((k,S)\) if \(0\le J\le S\) and \(rk+J\equiv0\pmod L\).  If \(J_0=\res{-rk}{L}\), then the admissible integers are \(J_0+uL\) lying in \([0,S]\).  In particular, if \(S<L\), there is at most one admissible integer; if it exists, then \(p^{1+v_p(S)}\mid\binom SJ\).
\end{proposition}
\begin{proof}
Let  $t=1+v_p(S)$ and $\widehat a$ be the lift of $a$, and let $G(X)=X^r(X^{q-1}+\widehat a)$. Applying Lemma \ref{lem:lifted_power_sum} to $G(X)$ yields that for any $1\le S<q^e-1$,
\begin{align*}
0=\sum_{x\in\mathcal T_t}G(x)^S
 &=\sum_{x\in\mathcal T_t}
 \sum_{0\le J\le S}\binom{S}{J}\widehat{a}^{S-J}x^{rS+(q-1)J}\\
 &=\sum_{0\le J\le S} \binom{S}{J}\widehat{a}^{S-J} \sum_{x\in\mathcal T_t}x^{rS+(q-1)J}\\
&=(q^e-1)\sum_{\substack{0\le J\le S\\ rk+J\equiv0\pmod L}}
\binom{S}{J}\widehat a^{\,S-J},
\end{align*}
where the last equality follows from \eqref{eq:power_sumTt}.
 Indeed,  $rS+(q-1)J=(q-1)(rk+J)$, so
$q^e-1\mid rS+(q-1)J$ if and only if $rk+J\equiv0\pmod L$.
Since $p\nmid q^e-1$, the element $q^e-1$ is a unit in
$\mathcal A_t$. Hence \eqref{eq:lifted-sum} is proved.

The congruence $rk+J\equiv0\pmod L$ is equivalent to
$J\equiv -rk\pmod L$. Thus, if $J_0=\res{-rk}{L}$, all admissible integers are of the
form $J_0+uL$ lying in $[0,S]$. In particular, if $S<L$,
there is at most one admissible integer. If such an admissible \(J\) exists, then \eqref{eq:lifted-sum} gives
$\binom{S}{J}\widehat a^{\,S-J}=0$ in $\mathcal A_t.$
Since $\widehat a$ is a unit and $t=1+v_p(S)$, we obtain
$p^{1+v_p(S)}\mid\binom{S}{J}.$
\end{proof}

Next, we apply Proposition~\ref{prop:lifted-binomial} to derive some necessary restrictions on the exponent of a permutation binomial $F_{r,a}$.

\begin{lemma}\label{lem2.3}
Assume $F_{r,a}(X)$ permutes $\mathbb F_{q^e}$, and let $\bar{r}=\langle r\rangle_L$. Then $1\le \bar r<L$.
 Moreover, for every integer \(1\le h<e\), we have $\langle-\bar rL_h\rangle_L\ge q^h>0.$  Equivalently, if $\bar rL_h=n_hL+\rho_h$ with $0\le \rho_h<L$, then $0<\rho_h\le L-q^h$. Furthermore, $\bar r\equiv1\pmod q.$ If \(e=2\), then $\bar r=1$.
\end{lemma}
\begin{proof}
First, \(\bar r\neq0\). Otherwise $r\equiv0\pmod L$. Taking $k=1$ and $S=q-1<L$ in Proposition \ref{prop:lifted-binomial}, the only admissible index is $J=0$. Therefore
$\binom {q-1} 0\widehat {a}^{\,q-1}=0$ in $\mathcal{A}_{1+v_p(S)}=\mathcal{A}_1$. This yields \(\widehat{a}^{q-1}=0\) in $\mathcal{A}_{1}$, impossible because $a\neq0$. Hence $1\le\bar r<L.$

Fix $1\le h<e$ and take $k=L_h$, so that $S=(q-1)k=q^h-1<L$.

\noindent{\bf Claim.} For the above choice $k=L_h$, there is no admissible index $J\in[0,S]$. 

\noindent{\it Proof of Claim.}
By Proposition \ref{prop:lifted-binomial},  if an admissible index $J\in[0,S]$ exists,  then $p^{1+v_p(S)}\mid\binom SJ$.  Since $v_p(S)=0$, this reduces to $p\mid\binom SJ$. On the other hand, since $S=p^{mh}-1$, in $\mathbb{F}_p[X]$, one has
$(1+X)^S=(1+X^{p^{mh}})/(1+X)=\sum_{J=0}^S(-1)^JX^J$. Comparison with the binomial expansion yields $\binom SJ\equiv(-1)^J\pmod p$.  Therefore $p\nmid\binom SJ$ for $0\le J\le S.$  $\square$ 

Therefore, $\langle-rk\rangle_L=\langle-\bar rL_h\rangle_L\ge q^h>0$.
In particular, write $\bar{r}L_h=n_hL+\rho_h$ with $0\le\rho_h<L$. Then $0<\rho_h<L$, because $\rho_h=0$ would imply $\langle -\bar r L_h\rangle_L=0$, contradicting $\langle -\bar r L_h\rangle_L\ge q^h>0$.

It remains to prove $\bar r\equiv1\pmod q$. Take $h=e-1$ and $J_0'=\langle-\bar rL_{e-1}\rangle_L.$ Then $J_0'\ge q^{e-1}$. Observe that \(qL_{e-1}=L-1\), so
$$qJ_0'\equiv -q\bar{r}L_{e-1}\!\pmod{L} = -\bar{r}(L-1)\equiv \bar{r}\!\pmod{L}.$$
 Thus, $qJ_0'=\bar r+uL$ from some $0\le u\leq q-1$,  since $0\le J_0'<L$ and $1\le\bar r<L$. If $u\le q-2$, then $qJ_0'\leq (L-1)+(q-2)L=q^e-2.$
This contradicts $qJ_0'\geq q^e$. Hence $u=q-1$, and thus $qJ_0'=\bar{r}+(q-1)L$.  Reducing modulo $q$, we get \(\bar{r}\equiv L\equiv1\pmod{q}\).

Finally, if $e=2$, then $h=1$ and $0<\rho_1\leq L-q=1$.  Since $\bar r<L$ and $\bar r=n_1L+\rho_1$, necessarily $n_1=0$, and thus $\bar r=\rho_1=1$.
\end{proof}

\section{Arithmetic reduction of the exponent condition}

\begin{proposition}\label{prop:quotient-divisibility}
Assume that $F_{r,a}$ permutes $\mathbb{F}_{q^e}$.  For $1\le h<e$, write uniquely $\bar rL_h=n_hL+\rho_h$ with $0\le\rho_h<L$. Then $q\mid n_h$ for every $1\le h<e$.
\end{proposition}

The remainder of this section is devoted to the proof of Proposition~\ref{prop:quotient-divisibility}.  The case $e=2$ is an immediate result from Lemma~\ref{lem2.3}, because then $\bar r=1$ and $n_1=0$.  We therefore assume \(e\ge3\) throughout the auxiliary arguments below.

\begin{definition}\label{def:bad_indices}
Let $e\ge3$ and $1\le\bar r<L$.  For each $1\le h<e$, write uniquely $\bar{r}L_h = n_h L + \rho_h$ with $0\le \rho_h < L$.  We call $h$  \emph{good} if $q\mid n_h$ and \emph{bad} otherwise, and set $\mathcal B:=\{\,h\in\{1,\ldots,e-1\}:q\nmid n_h\,\}.$
\end{definition}
%

In this terminology, Proposition~\ref{prop:quotient-divisibility} amounts to proving that \(\mathcal B=\varnothing\).  We argue by contradiction.  Assuming \(\mathcal B\ne\varnothing\), we first analyze the arithmetic structure of a general bad parameter and identify two additional properties under which it is incompatible with the permutation property of \(F_{r,a}\).  We then show that if $\mathcal{B}\neq\varnothing$, the last bad \(h\in\mathcal B\) necessarily has both properties.

The following standing assumptions are direct consequences of Lemma~\ref{lem2.3} and will be used throughout the rest of this section.

\begin{hypothesis}\label{hyp}
Assume $e\ge3$, $1\le\bar r<L$, and $\bar{r}\equiv1\pmod q$.  Set $n_0=\rho_0=0$, and for $1\le h<e$ write $\bar rL_h=n_hL+\rho_h$, where $0<\rho_h\le L-q^h$.
\end{hypothesis}

\subsection{Bad parameters and their arithmetic structure}

Since $L_1=1$ and $1\le\bar r<L$, we have $n_1=0$, so every bad parameter satisfies $h\ge2$. To analyze a bad parameter, we compare the Euclidean divisions of $\bar{r}L_{h-1}$ and $\bar{r}L_h$ by $L$. Since $L_h=qL_{h-1}+1$, the relations $\bar rL_{h-1}=n_{h-1}L+\rho_{h-1}$ and $\bar rL_h=n_hL+\rho_h$ yield $q\rho_{h-1}+\bar r=(n_h-qn_{h-1})L+\rho_h$.
We therefore introduce
\begin{equation}\label{def:delta}
\delta_h=n_h-qn_{h-1}
=\left\lfloor\frac{q\rho_{h-1}+\bar r}{L}\right\rfloor.
\end{equation}
Thus, $\delta_h$ measures the change in the quotient when passing from the division of $\bar rL_{h-1}$ by $L$ to that of $\bar rL_h$ by $L$.

\begin{lemma}\label{lem:delta}
Let $1\le h<e$. Then $h$ is good if and only if $\delta_h\in\{0,q\}$. Equivalently, $h$ is bad if and only if $1\le\delta_h\le q-1$.
\end{lemma}
\begin{proof}
 Since $\bar{r}=\langle r\rangle_L$ and $\rho_{h-1}=\langle\bar{r}L_{h-1}\rangle_L$, by \eqref{def:delta}, we have $0\le\delta_h\le q$. Moreover, $\delta_h=n_h-qn_{h-1}$ implies $\delta_h\equiv n_h\pmod q$. Therefore $q\mid n_h$ if and only if $q\mid\delta_h$, which, together with $0\le\delta_h\le q$, is equivalent to
$\delta_h\in\{0,q\}$. The complementary case is precisely
$1\le\delta_h\le q-1$.
\end{proof}

\begin{lemma}\label{lem:bad-parameter}
Assume Hypothesis~\ref{hyp} holds, and let \(h\in\mathcal B\).
Define $J_h=\left\langle-\bar{r}q^{h-1}L_{e-h}\right\rangle_L$.
There exists a unique integer $\ell_h=(\bar{r}-1)/q-n_{h-1}$ such that
\begin{align*}
&J_h=\ell_hL-\bar r q^{h-1}L_{e-h}
=\frac{\rho_h+(\delta_h-1)L}{q},\\
&0<J_h<(q^{e-h}-1)q^{h-1}<L.
\end{align*}
Let $J_h=Q_hq^{h-1}+R_h$ with $0\le R_h<q^{h-1}$. Then $(\delta_h-1)L_{e-h}\le Q_h<\delta_hL_{e-h}$.

Furthermore, there exists a unique integer \(t_h\) such that $\ell_h=n_{e-h}q^{h-1}+t_h$ with $1\leq t_h<q^{h-1}$.
Let $\nu_h\in\{1,\ldots,q-1\}$ be the unique integer satisfying $\nu_h\equiv t_h\pmod{q-1}$. Then
\[
R_h=\frac{\nu_hq^{h-1}-t_h}{q-1},
\qquad\text{and}\qquad
Q_h+1\equiv n_{e-h}+\nu_h\pmod q.\]
In particular, if $e-h$ is good, then
$v_p(Q_h+1)=v_p(\nu_h)<m.$
\end{lemma}
\begin{proof}

We first determine $J_h$. By definition,
\(J_h\) is the unique integer in \([0,L)\) satisfying
$J_h\equiv-\bar r q^{h-1}L_{e-h}\pmod L$. Thus it suffices to find
an integer \(\ell_h\) for which $0<\ell_h L-\bar{r}q^{h-1}L_{e-h}<L.$ We claim that
\(\ell_h=(\bar r-1)/q-n_{h-1}\) has this property.
It is an integer
\(\bar r\equiv1\pmod q\). Notice that \(L=q^hL_{e-h}+L_h\), it follows that
\begin{align*}
q\bigl(\ell_hL-\bar r q^{h-1}L_{e-h}\bigr)
&=(\bar r-1-qn_{h-1})L-\bar r q^hL_{e-h}\\
&=\bar rL_h-(qn_{h-1}+1)L\\
&=(n_h-qn_{h-1}-1)L+\rho_h\\
&=(\delta_h-1)L+\rho_h.
\end{align*}
Therefore $\ell_hL-\bar{r}q^{h-1}L_{e-h}=\left(\rho_h+(\delta_h-1)L\right)/q$, which is positive since \(\delta_h\ge 1\) by Lemma~\ref{lem:delta}.
On the other hand, using \(\rho_h\le L-q^h\), \(L=q^hL_{e-h}+L_h\), \(\delta_h\le q-1\), and \((q-1)L_h=q^h-1\), we obtain
\begin{align*}
\ell_hL-\bar r q^{h-1}L_{e-h}
\le \frac{\delta_hL-q^h}{q}
&=\delta_hq^{h-1}L_{e-h}+\frac{\delta_hL_h}{q}-q^{h-1}\\
&<\delta_hq^{h-1}L_{e-h}\\
&\le(q^{e-h}-1)q^{h-1}<L.
\end{align*}
Hence this integer is exactly $J_h$, and the stated bounds for $J_h$ holds.

Now let $J_h=Q_hq^{h-1}+R_h$ with $0\le R_h<q^{h-1}$, that is, $Q_h=\lfloor J_h/q^{h-1}\rfloor$. Since $\rho_h>0$ and $L=q^hL_{e-h}+L_h>q^hL_{e-h}$, we have
$J_h>(\delta_h-1)q^{h-1}L_{e-h}$, while the preceding upper bound gives $J_h<\delta_hq^{h-1}L_{e-h}$.  Dividing by $q^{h-1}$ and taking floors yields
$(\delta_h-1)L_{e-h}\le Q_h<\delta_hL_{e-h}$.

To obtain $t_h$, use $J_h=\ell_hL-\bar r q^{h-1}L_{e-h}$ and $\bar{r}L_{e-h}=n_{e-h}L+\rho_{e-h}$.  we obtain
$$\frac{\ell_h}{q^{h-1}}
=n_{e-h}+\frac{\rho_{e-h}+J_h/q^{h-1}}{L}.$$
Because $0<\rho_{e-h}\le L-q^{e-h}$ and $0<J_h<(q^{e-h}-1)q^{h-1}$, so the last term lies strictly between $0$ and $1$. Hence
$\ell_h=n_{e-h}q^{h-1}+t_h$ for a unique integer $t_h$ satisfying $1\leq t_h<q^{h-1}$.

Reducing $J_h=\ell_hL-\bar r q^{h-1}L_{e-h}$ modulo $q^{h-1}$ gives
$R_h\equiv t_hL\pmod{q^{h-1}}$. As $L\equiv1\pmod p$, one has $\gcd(L,q^{h-1})=1$. Thus $R_h\ne0$; otherwise $q^{h-1}\mid t_h$, contradicting $1\leq t_h<q^{h-1}$.

Let $\nu_h\in\{1,\ldots,q-1\}$ be the unique integer satisfying $\nu_h\equiv t_h\pmod{q-1}$. From the two representations of $J_h$, we have $R_h\equiv\ell_hL\pmod{q^{h-1}}$, and thus
$q^{h-1}\mid \ell_h+(q-1)R_h$ since $(q-1)L=q^e-1\equiv-1\pmod{q^{h-1}}$. Put
$B=\left(\ell_h+(q-1)R_h\right)/q^{h-1}$. Using $\ell_h=n_{e-h}q^{h-1}+t_h$ together with $0<t_h,R_h<q^{h-1}$ gives $n_{e-h}<B<n_{e-h}+q$. Moreover, $q^{h-1}\equiv1\pmod{q-1}$, and hence
$B\equiv\ell_h\equiv n_{e-h}+t_h\equiv n_{e-h}+\nu_h\pmod{q-1}$. Since both $B-n_{e-h}$ and $\nu_h$ lie in $\{1,\ldots,q-1\}$, we obtain $B=n_{e-h}+\nu_h$. Therefore
$\ell_h+(q-1)R_h=(n_{e-h}+\nu_h)q^{h-1}$, and substitution of $\ell_h=n_{e-h}q^{h-1}+t_h$ yields
$R_h=(\nu_hq^{h-1}-t_h)/(q-1)$.

Finally, substitute this identity into
$J_h=Q_hq^{h-1}+R_h=\ell_hL-\bar r q^{h-1}L_{e-h}$. Multiplying by $q-1$ and using $(q-1)L=q^e-1$, $(q-1)L_{e-h}=q^{e-h}-1$, and $t_h=\ell_h-n_{e-h}q^{h-1}$ gives
$$(q-1)Q_h+n_{e-h}+\nu_h
=\ell_hq^{e-h+1}-\bar r q^{e-h}+\bar r.$$
Reducing modulo $q$ and using $\bar r\equiv1\pmod q$ gives
$Q_h+1\equiv n_{e-h}+\nu_h\pmod q$.

If $e-h$ is good, then $q\mid n_{e-h}$, so $Q_h+1\equiv\nu_h\pmod{p^m}$. Since $1\le\nu_h\le p^m-1$, we have $v_p(\nu_h)<m$, and the congruence implies $v_p(Q_h+1)=v_p(\nu_h)<m$.
\end{proof}

\subsection{Exclusion of a bad parameter}


The preceding lemma provides some arithmetic characterizations associated with a bad $h\in\mathcal{B}$,  but does not by itself rule out its
existence. We next show that, under two further conditions, such a bad parameter cannot occur when $F_{r,a}$ is a permutation.

\begin{proposition}\label{prop:exclusion}
Under Hypothesis~\ref{hyp}, let $h\in\mathcal B$ and let $J_h,Q_h,R_h,t_h,\nu_h$ be as in Lemma~\ref{lem:bad-parameter}.  Suppose that $e-h$ is good and $(\nu_h-1)L_{e-h}\le Q_h<\nu_hL_{e-h}.$
Then there exists no $a\in\mathbb{F}_{q^e}^{*}$ such that $F_{r,a}$ permutes $\mathbb{F}_{q^e}$.
\end{proposition}
\begin{proof}
 Let $w=v_p(\nu_h)$ and assume $e-h$ is good. By Lemma~\ref{lem:bad-parameter}, $v_p(Q_h+1)=w<m$.  Suppose for contradiction that $F_{r,a}$ is a permutation.  Apply Proposition~\ref{prop:lifted-binomial} with
$k=p^wq^{h-1}L_{e-h}$ and $S=p^w(q^{e-h}-1)q^{h-1}$. Then $S<q(q^{e-1}-q^{h-1})<q^{e}-1$, as required.
Moreover, we have $v_p(S)=w+m(h-1)$ because $p\nmid q^{e-h}-1$. We thus consider \eqref{eq:lifted-sum} in $\mathcal{A}_{t}$ with $t=1+w+m(h-1)$.

We use the following three claims to complete the proof.

\noindent{\bf Claim 1.} Every admissible index lying in $[0,S]$ is of the form $J_u=p^wJ_h+uL$ with $|u|<p^w$.

\noindent{\it Proof of Claim 1.} Since
$r\equiv\bar r\pmod L$ and $J_h=\res{-\bar{r}q^{h-1}L_{e-h}}{L}$,
we have $-rk\equiv p^wJ_h \pmod L.$
Therefore every admissible index lying in $[0,S]$ is of the form $J_u=p^wJ_h+uL$, where $u\in\mathbb{Z}$.
By Lemma~\ref{lem:bad-parameter},
$0<J_h<(q^{e-h}-1)q^{h-1}<L$, and hence $S<p^wL$. Consequently, $|u|<p^w$. Indeed, if $u\le -p^w$, then $J_u\le p^w(J_h-L)<0$, while if
$u\ge p^w$, then $J_u>p^wL>S$. $\square$   

\medskip
\noindent{\bf Claim 2.} $J_0=p^wJ_h$ is admissible, and
$\binom{S}{J_0}\ne0$ in $\mathcal A_t$.

\noindent{\it Proof of Claim 2.}
From the proof of Claim 1, we see that $-rk\equiv p^wJ_h=J_0 \pmod L$. Moreover, $0<J_h<(q^{e-h}-1)q^{h-1}$ implies $0<p^wJ_h<p^w(q^{e-h}-1)q^{h-1}=S$, i.e., $0<J_0<S$. Thus $J_0$ is admissible.  Next, we consider $\binom{S}{J_0}$ in $\mathcal{A}_t$. Since $S=p^{\,t-1}\bigl(p^{m(e-h)}-1\bigr)$ and
$$J_0=p^wJ_h
=p^wq^{h-1}Q_h+p^wR_h
=p^{t-1}Q_h+p^wR_h,$$
with $0<p^wR_h<p^wq^{h-1}=p^{t-1}$,
Lemma~\ref{lem:kummer_special} gives
\begin{align*}
v_p\binom{S}{J_0}
&=t-1+v_p(Q_h+1)-v_p(p^wR_h)\\
&=t-1+w-\bigl(w+v_p(R_h)\bigr)\\
&=t-1-v_p(R_h)<t.
\end{align*}
Therefore $p^t\nmid \binom{S}{J_0}$, that is, $\binom{S}{J_0}\ne0$ in $\mathcal A_t$. $\square$

\medskip
\noindent{\bf Claim 3.} For each admissible $J_u$ with $u\ne0$, $\binom{S}{J_u}=0$ in $\mathcal{A}_{t}$.

\noindent{\it Proof of Claim 3.}
 If $w=0$, Claim~1 gives $u=0$, and there is nothing to prove. Thus assume $w>0$.
 It suffices to show that $v_p\binom{S}{J_u}\geq t$. Using the relations $L=q^hL_{e-h}+L_h$ and $J_h=Q_hq^{h-1}+R_h$, we may write
$$J_u=p^wq^{h-1}Q_u+E_u=p^{t-1}Q_u+E_u,$$
where $Q_u=Q_h+uqL_{e-h}/p^w$ and $E_u=p^wR_h+uL_h$.
Since $w<m$, both $Q_u$ and $E_u$ are integers. Moreover, the assumption $(\nu_h-1)L_{e-h}\le Q_h<\nu_hL_{e-h}$
gives
\begin{equation}\label{eq:Q_u}
\left(\nu_h+\frac{uq}{p^w}-1\right)L_{e-h}
\le Q_u<
\left(\nu_h+\frac{uq}{p^w}\right)L_{e-h}.
\end{equation}
By Lemma~\ref{lem:bad-parameter}, $(q-1)R_h=\nu_hq^{h-1}-t_h$, together with $(q-1)L_h=q^h-1$, we obtain
\begin{equation}\label{eq:E_u}
(q-1)E_u=p^wq^{h-1}\left(\nu_h+\frac{u q}{p^w}\right)
-(p^wt_h+u).
\end{equation}
Since $1\leq t_h<q^{h-1}$ and $|u|<p^w$ are integers, $1\le p^wt_h+u\le p^wq^{h-1}-1$.
 On the other hand, it is straightforward to check that $1\le \nu_h+\frac{q}{p^w}u\le q-1$.
Indeed ,if $\nu_h+qu/p^w\le0$, then \eqref{eq:Q_u} and \eqref{eq:E_u} imply
$Q_u<0$ and $E_u<0$, and hence $J_u<0$, a contradiction.
If $\nu_h+qu/p^w\ge q$, then $Q_u\ge(q-1)L_{e-h}=q^{e-h}-1$ and $E_u>0$, so $J_u>p^wq^{h-1}(q^{e-h}-1)=S$, again a contradiction. Therefore, \eqref{eq:E_u} implies $0<E_u<p^wq^{h-1}=p^{\,t-1}$.

Therefore, by $S=p^{\,t-1}\bigl(p^{m(e-h)}-1\bigr)$ and $J_u=p^{t-1}Q_u+E_u$ with  $0<E_u<p^{\,t-1}$, Lemma \ref{lem:kummer_special} gives that
\begin{align*}
v_p\binom{S}{J_u}&=t-1+v_p(Q_u+1)-v_p(E_u)\\
&=t-1+v_p\left(Q_h+1+\frac{up^mL_{e-h}}{p^w}\right)-v_p\left(p^wR_h+uL_h\right).
\end{align*}
Since $u\ne0$ and $|u|<p^w$, we have
$v_p(u)<w$. As $L_h\equiv1\pmod p$, $v_p(p^wR_h+uL_h)=v_p(u)$.
Moreover, by $v_p(Q_h+1)=w$, $p\nmid L_{e-h}$, and $v_p\left(u qL_{e-h}/p^w\right)
=m-w+v_p(u)\ge1+v_p(u)$, we have $v_p(Q_u+1)\ge 1+v_p(u)$. Hence
$$v_p\binom{S}{J_u}\geq t-1+(1+v_p(u))-v_p(u)=t$$
Thus, for every admissible $J_u$ with $u\neq0$, $p^t\mid \binom{S}{J_u}=0$, so its image in $\mathcal A_t$. $\square$  

By Claim~3, every term in \eqref{eq:lifted-sum}
with $u\ne0$ vanishes in $\mathcal A_t$. Claim~2 shows that the coefficient of the term with $u=0$ does not vanish.
Since $a\neq0$ in $\mathbb{F}_{q^e}$, its lift $\widehat a$ is a unit in $\mathcal A_t$, and thus
$\binom{S}{J_0}\widehat a^{\,S-J_0}\ne0$ in $\mathcal A_t.$
This contradicts \eqref{eq:lifted-sum}. Hence \(F_{r,a}\) cannot be a
permutation of \(\mathbb F_{q^e}\).
\end{proof}

\subsection{Elimination of bad parameters}

In this subsection, assume that $\mathcal B\neq\varnothing$. The next two results together show that the maximal bad $h\in\mathcal B$ necessarily satisfies the two additional conditions in Proposition~\ref{prop:exclusion}.

\begin{proposition}\label{prop:maxB1}
Assume Hypothesis~\ref{hyp} and $\mathcal B\neq\varnothing$. Let $s=\min \mathcal B$ and $b=\max \mathcal B$ denote the smallest and largest elements of $\mathcal B$, respectively.
Then $s+b>e$.  In particular, $e-b$ is good.
\end{proposition}

\begin{proof}
First note that $s\geq2$ since $n_1=0$. Assume for contradiction that $s+b\le e$. Then $1,\ldots,s-1$ and $e-s+1,\ldots,e-1$ are all good because $b\leq e-s$. We next consider the arithmetic properties of the good parameters at the two ends.

Since $(q-1)L_h=q^h-1$, the relation
$\bar rL_h=n_hL+\rho_h$ gives $(\bar r+q-1)L_h=n_hL+\rho_h+q^h-1$.
Now denote $\rho_h^+=\rho_h+q^h-1$ for $0\le h<e$, and so $(\bar r+q-1)L_h=n_hL+\rho_h^+$.
By Hypothesis~\ref{hyp}, $\rho_0^+=0$, while
$0<\rho_h^+<L$ for $1\le h<e$. Thus the Euclidean divisions of $\bar rL_h$ and $(\bar r+q-1)L_h$ by $L$ have the same quotient $n_h$, with remainders $\rho_h$ and $\rho_h^+$, respectively.

 By $\delta_h=n_h-qn_{h-1}$ and $L_h=qL_{h-1}+1$, we derive the recurrence relations for $\rho_h$ and $\rho_h^+$:
\begin{equation}\label{eq:rho_h}
\begin{aligned}
\rho_h=\bar{r}L_h-n_hL&=q\bar{r}L_{h-1}+\bar{r}-n_hL\\
 &=q\left(n_{h-1}L+\rho_{h-1}\right)+\bar{r}-n_hL\\
 &=q\rho_{h-1}+\bar r-\delta_hL.
\end{aligned}
\end{equation}
Similarly, together with $(\bar r+q-1)L_{h-1}=n_{h-1}L+\rho_{h-1}^+$, we obtain
\begin{align}\label{eq:rho_h+}
\rho_h^+=q\rho_{h-1}^++\bar r+q-1-\delta_hL.
\end{align}

We next consider the case $h=e-1$.
Since $qL_{e-1}=L-1$, we have $\bar rL_{e-1}=(\bar r-1)L/q+(L-\bar r)/q$, where the second term is exactly $\res{\bar{r}L_{e-1}}{L}$, as $L\equiv\bar r\equiv1\pmod q$ and $0<L-\bar r<L$.
Comparing with $\bar rL_{e-1}=n_{e-1}L+\rho_{e-1}$ gives $qn_{e-1}=(\bar r-1)$ and $q\rho_{e-1}=L-\bar r$.
Since $\rho_{e-1}^+=\rho_{e-1}+q^{e-1}-1$, we also obtain
$q\rho_{e-1}^+=qL-\bar r-(q-1)$.

For $1\le i\le s=\min\mathcal{B}$, set $u_i=\rho_{e-i}-\rho_{s-i}$ and $u_i^+=\rho_{e-i}^+-\rho_{s-i}^+$.
Then \(-L<u_i,u_i^+<L\). Using the relations above together with the \eqref{eq:rho_h} and \eqref{eq:rho_h+}, we obtain
$qu_1=(1-\delta_s)L-\rho_s$ and $qu_1^+=(q-\delta_s)L-\rho_s^+$.
Since $s$ is bad, Lemma~\ref{lem:delta} shows
$1\le\delta_s\le q-1$, and hence $u_1<0<u_1^+$.

Now fix $1\le i\le s-1$. Then both $s-i$ and $e-i$ are good,  so Lemma~\ref{lem:delta} yields
$\left(\delta_{e-i}-\delta_{s-i}\right)/q\in\{-1,0,1\}.$
Applying  \eqref{eq:rho_h} and \eqref{eq:rho_h+} at \(h=e-i\) and
\(h=s-i\), and then subtracting, gives
$$u_{i+1}=\frac{u_i}{q}+\frac{\delta_{e-i}-\delta_{s-i}}{q}L,
\qquad \text{and}\qquad u_{i+1}^+=\frac{u_i^+}{q}+\frac{\delta_{e-i}-\delta_{s-i}}{q}L.$$
If \(u_i<0<u_i^+\), then
\((\delta_{e-i}-\delta_{s-i})/q=1\) would imply
\(u_{i+1}^+>L\), while
\((\delta_{e-i}-\delta_{s-i})/q=-1\) would imply
\(u_{i+1}<-L\). Hence
\(\delta_{e-i}=\delta_{s-i}\) for \(1\le i\le s-1\), and therefore
\(u_{i+1}<0<u_{i+1}^+\).
Starting from \(u_1<0<u_1^+\), induction yields \(u_s<0\).
However, $u_s=\rho_{e-s}-\rho_0=\rho_{e-s}>0$,
a contradiction. Hence \(s+b>e\). Consequently,
\(e-b<s=\min\mathcal B\), so \(e-b\) is good.
\end{proof}

\begin{proposition}\label{prop:maxB2}
Assume Hypothesis~\ref{hyp}, $\mathcal B\neq\varnothing$, and let $s=\min\mathcal B$, $b=\max\mathcal B$. Then $\nu_b=\delta_b$.
Therefore $(\nu_b-1)L_{e-b}\le Q_b<\nu_bL_{e-b}$.
\end{proposition}
\begin{proof}
By Proposition~\ref{prop:maxB1}, $e-b<s$. Hence
$1,\ldots,e-b$ and $b+1,\ldots,e-1$ are all good. The following compares the corresponding values of $\delta_h$ at these two ends.

Now let $0\le i\le e-b-1$, set $d_i=\rho_{b+i}-\rho_i$.
Then $d_0=\rho_b>0$ and $-L<d_i<L$ for all such $i$. Moreover, for fixed $1\le i\le e-b-1$, both $i$ and $b+i$ are good, so Lemma~\ref{lem:delta} gives
$\left(\delta_{b+i}-\delta_i\right)/q\in\{-1,0,1\}$.
Applying \eqref{eq:rho_h} at \(h=b+i\) and \(h=i\), and
subtracting the resulting identities, we obtain
$$d_i=q\left(d_{i-1}
-\frac{\delta_{b+i}-\delta_i}{q}L\right), \quad \text{for }1\leq i\le e-b-1.$$
We claim that none of the $d_i$ is zero.
Indeed, if $d_i=0$, then
$d_{i-1}=\frac{\delta_{b+i}-\delta_i}{q}L$. Since $-L<d_{i-1}<L$, it follows that $d_{i-1}=0$.
Iterating backwards eventually yields $d_0=0$, contradicting $d_0>0$.

Suppose  that
$(\delta_{b+i}-\delta_i)/q=1$. Then
$d_i=q(d_{i-1}-L)<0$, while $d_i>-L$ implies
$d_{i-1}>0$. Hence $d_{i-1}>0>d_i$.
The other two cases are analogous. Therefore, for $1\le i\le e-b-1$,
\begin{equation*}
\frac{\delta_{b+i}-\delta_i}{q}=
\begin{cases}
1, & d_{i-1}>0>d_i,\\
0, & d_{i-1}d_i>0, \\
-1, & d_{i-1}<0<d_i.
\end{cases}
\end{equation*}
That is, the term is nonzero only when $d_i$ changes sign.
Since $d_0>0$, we have $\sum_{i=1}^{e-b-1}\frac{\delta_{b+i}-\delta_i}{q}$
equals $0$ if $d_{e-b-1}>0$, and $1$ if $d_{e-b-1}<0$. 
Equivalently, \begin{equation}\label{eq:sum1}
\sum_{i=1}^{e-b-1}
(\delta_{b+i}-\delta_i)
=
\begin{cases}
0, & d_{e-b-1}>0,\\
q, & d_{e-b-1}<0.
\end{cases}
\end{equation}

On the other hand, using $q\rho_{e-1}=L-\bar r$ and \eqref{eq:rho_h} at $h=e-b$, we obtain $qd_{e-b-1}=q\rho_{e-1}-q\rho_{e-b-1}=(1-\delta_{e-b})L-\rho_{e-b}.$
Since $e-b$ is good, Lemma~\ref{lem:delta} gives
$\delta_{e-b}\in\{0,q\}$. Hence $d_{e-b-1}>0$ if $\delta_{e-b}=0$, whereas $d_{e-b-1}<0$ if $\delta_{e-b}=q$. Combining this with \eqref{eq:sum1}, we obtain $\sum_{i=1}^{e-b-1}
\bigl(\delta_{b+i}-\delta_i\bigr)
=\delta_{e-b}.$
Equivalently,
\begin{equation}\label{eq:sum2}
\sum_{h=b+1}^{e-1}\delta_h
=\sum_{h=1}^{e-b}\delta_h.\end{equation}
Since
$\delta_h=n_h-qn_{h-1}\equiv n_h-n_{h-1}\pmod{q-1}$,
\eqref{eq:sum2} gives $n_{e-1}-n_b-n_{e-b}\equiv0\pmod{q-1}$.
From the proof of Proposition~\ref{prop:maxB1},
$n_{e-1}=(\bar r-1)/q$, and hence
$$\frac{\bar r-1}{q}-n_b-n_{e-b}\equiv0\pmod{q-1}.$$
By Lemma~\ref{lem:bad-parameter}, at $h=b\in\mathcal{B}$, we have $\nu_b\equiv t_b\pmod{q-1}$ and
$$\frac{\bar r-1}{q}-n_{b-1}
=n_{e-b}q^{b-1}+t_b.$$
Since \(q^{b-1}\equiv1\pmod{q-1}\) and
\(\delta_b\equiv n_b-n_{b-1}\pmod{q-1}\), it follows that
\begin{align*}
\nu_b-\delta_b&\equiv
\frac{\bar r-1}{q}-n_{b-1}-n_{e-b}
-\bigl(n_b-n_{b-1}\bigr)\\
&=\frac{\bar r-1}{q}-n_b-n_{e-b}\\
&\equiv0\pmod{q-1}.
\end{align*}
Hence $\nu_b=\delta_b$, since both $\nu_b$ and $\delta_b$ belong to $\{1,\dots, q-1\}$.

Finally, Lemma~\ref{lem:bad-parameter} at $h=b$  also gives
$(\delta_b-1)L_{e-b}\le Q_b<\delta_bL_{e-b}.$ Substituting $\nu_b=\delta_b$ yields
$(\nu_b-1)L_{e-b}\le Q_b<\nu_bL_{e-b}.$
\end{proof}

\medskip
{\it Proof of Proposition \ref{prop:quotient-divisibility}.}
Suppose that $F_{r,a}$ is a permutation of $\mathbb{F}_{q^e}$ and $\mathcal B\neq\varnothing$, and let
$b=\max\mathcal B$. By Proposition~\ref{prop:maxB1},
$e-b$ is good. Proposition~\ref{prop:maxB2} further gives
$(\nu_b-1)L_{e-b}\le Q_b<\nu_bL_{e-b}$.
Thus $b$ satisfies both conditions of
Proposition~\ref{prop:exclusion}, which contradicts the permutation
property of $F_{r,a}$ . Therefore
$\mathcal B=\varnothing$.
Equivalently, $q\mid n_h$ for every $1\le h<e$.

\section{From quotient divisibility to the exponent congruence}

The preceding section shows $q\mid n_h$ holds for every $1\le h<e$. We now translate this divisibility into the exponent congruence stated in Theorem~\ref{thm:main}.

\begin{proposition}\label{prop:exponent-condition}
Let $q\ge 2$ and $e\ge2$ be integers. Suppose that $1\le\bar r<L$, $\bar r\equiv1\pmod q$. For $1\le k<e$, write $\bar{r}L_k=n_kL+\rho_k$ with $0\le\rho_k<L$. If  $q\mid n_k$ for every $1\le k<e$, then there exists $1\le h<e$ with $\gcd(h,e)=1$ such that $r(q^h-1)\equiv q-1\pmod{q^e-1}$. Equivalently, $rL_h\equiv1\pmod L$.
\end{proposition}

\begin{proof}
Write $\bar r=\sum_{i=0}^{e-1}r_iq^i$ with $0\le r_i<q$,
and regard $r_{e-1}\cdots r_0$ as a word of length $e$,  including leading zeros. Since $\bar{r}\equiv1\pmod q$, we have $r_0=1$. For $1\le k<e$, denote $\bar{r}^{(k)}$ as the integer from a left cyclic shift of this word by $k$ places.  Since $\bar r<L\le q^e-1$, its digits are not all equal to $q-1$, and cyclic shifts preserve this condition. Hence $0\le \bar r^{(k)}<q^e-1$.

We first claim that for $0\le i<e$, $r_i\in\{0,1\}$.
Fix $1\le k<e$, and write $\bar{r}=Aq^{e-k}+B$, where $A=\left\lfloor\bar{r}/q^{e-k}\right\rfloor$ and $0\le B<q^{e-k}.$ Then $\bar{r}^{(k)}=Bq^k+A$, and hence
$q^k\bar r=Aq^e+Bq^k=A(q^e-1)+\bar{r}^{(k)}$.
Therefore $\bar r(q^k-1)=A(q^e-1)+\bar{r}^{(k)}-\bar r.$
Since $n_k=\lfloor\bar{r}L_k/L\rfloor=\lfloor\bar r(q^k-1)/(q^e-1)\rfloor$ and $|\bar{r}^{(k)}-\bar r|<q^e-1$, we obtain
\begin{equation*}
n_k=\begin{cases}
A-1, & \bar r^{(k)}<\bar r,\\[2mm]
A,    & \bar r^{(k)}\ge\bar r.
\end{cases}
\end{equation*}
Since $q\mid n_k$, it follows that $q\mid A-1$ in the first
case and $q\mid A$ in the second. Moreover,
 using $A=\lfloor\bar{r}/q^{e-k}\rfloor\equiv r_{e-k}\pmod q$ and $0\le r_{e-k}<q$, we conclude that
\begin{equation}\label{eq:r_e-k}
r_{e-k}=\begin{cases}
1, & \bar r^{(k)}<\bar r,\\
0, & \bar r^{(k)}\ge\bar r.
\end{cases}
\end{equation}
Thus $r_i\in\{0,1\}$ for $1\le i<e$; together with $r_0=1$,
every base-$q$ digit of $\bar r$ belongs to $\{0,1\}$.

Next, we analyze the cyclic shifts and determine the positions for which $r_i=1$.

Set $\bar r^{(0)}=\bar r$. We claim that the cyclic shifts are pairwise distinct; namely, $\bar r^{(k_1)}\ne \bar r^{(k_2)}$ for $0\le k_1<k_2<e$.  Suppose that
$\bar r^{(k_1)}=\bar r^{(k_2)}$. Then
$\bar r^{(k_2-k_1)}=\bar r$. Let
$d=\gcd(k_2-k_1,e)$. Since some integer multiple of $k_2-k_1$ is congruent to $d$ modulo $e$, we also have
$\bar r^{(d)}=\bar r$.
Since $1\le d<e$, \eqref{eq:r_e-k} gives
$r_{e-d}=0$. On the other hand, $\bar r^{(d)}=\bar r$
implies that the digit word is fixed by the shift by $d$, and
so $r_{e-d}=r_0=1$, a contradiction. Hence the $e$ cyclic shifts are pairwise distinct.

Let $w=\sum_{i=0}^{e-1}r_i$.
Since $r_0=1$ and the all-ones word represents $L$, which is excluded by $\bar{r}<L$, we have $1\le w<e$.
Arrange the $e$ cyclic shifts in increasing order and label them $0,\ldots,e-1$. As the shifts run through all positions, each digit $r_i$ occurs exactly once as the last digit.  Hence exactly $w$ shifts end in $1$. By~\eqref{eq:r_e-k}, every nontrivial shift ending in $1$ is smaller than $\bar r$, whereas every shift ending in $0$ is larger than $\bar r$. Since $\bar r$ itself ends in $r_0=1$, it is labeled $w-1$. Thus the shifts labeled
$0,\ldots,w-1$ end in 1, while those labeled
$w,\ldots,e-1$ end in 0.

Similarly, exactly $e-w$ shifts begin with $0$ and $w$ shifts
begin with $1$. Specifically, the
former have labels $0,\ldots,e-w-1$, while the latter have labels $e-w,\ldots,e-1$.

We claim that the left cyclic shift preserves the relative order among the shifts beginning with 0, and likewise among those beginning with 1. Indeed, if two shifts with the same first digit $c\in\{0,1\}$ are written as $cq^{e-1}+X<cq^{e-1}+Y$, then $X<Y$. After moving the first digit to the end, they become $qX+c<qY+c$.
Therefore, one left cyclic shift sends a label $j$ to
$j+w$ if $0\le j<e-w$, and to $j+w-e$ if $e-w\le j<e$.
That is, modulo $e$, one left cyclic shift increases the label by $w$. After $t$ left cyclic shifts, the label therefore increases by $tw$ modulo $e$.

Recall that the $e$ cyclic shifts are pairwise distinct, the original word first returns to itself after $e$ left cyclic shifts. Hence the least positive integer $t$ satisfying $tw\equiv0\pmod e$ is $e$. Therefore $\gcd(w,e)=1$.

Choose $1\le h<e$ such that $wh\equiv1\pmod e$. In particular, \(\gcd(h,e)=1\).
Since $\bar r$ is labeled $w-1$, after $k$ left cyclic shifts its label is congruent to $w-1+kw\pmod e$.
Now fix $0\le i<e$. After $e-i$ left cyclic shifts, the digit $r_i$ becomes the last digit. The corresponding label is congruent to $w-1-iw\pmod e$. Since precisely those labels $0,\ldots,w-1$ correspond to shifts ending in 1, we have $r_i=1$ if and only if the residue of $w-1-iw$ modulo~$e$ lies in $\{0,\ldots,w-1\}$. This is equivalent to requiring that the residue of $iw$ modulo
$e$ belong to $\{0,\ldots,w-1\}$, that is, there exists an integer $t$ with $0\le t<w$ such that $iw\equiv t\pmod e$.  Since $wh\equiv1\pmod e$, it follows that $r_i=1$ if and only if $i\equiv th\pmod e$ for some $0\le t<w$.

Consequently, we have
$\bar{r}=\sum_{r_i=1}q^i
\equiv\sum_{t=0}^{w-1}q^{th}
\pmod{q^e-1}.$
Multiplying by $q^h-1$ yields
\begin{align}\label{eq:condition3}
(q^h-1)\bar r\equiv
(q^h-1)\sum_{t=0}^{w-1}q^{th}=q^{wh}-1
\equiv q-1
\pmod{q^e-1},
\end{align}
where the last congruence follows from $wh\equiv1\pmod e$.

Finally, from the definition of $L_h$ and $L$,  \eqref{eq:condition3} is equivalent to $rL_h\equiv1\pmod{L}$.
\end{proof}

\section{Proof of Theorem \ref{thm:main}}

With the preceding results, we are now ready to prove our main theorem.

\noindent{\it Proof of Theorem \ref{thm:main}.}
Suppose first that $F_{r,a}$ permutes $\mathbb F_{q^e}$.
By Proposition~\ref{conditions_for_a,q}, we have $(-a)^L\ne1$ and $\gcd(r,q-1)=1$.

If $e=2$, Lemma~\ref{lem2.3} gives $\bar r=1$. Taking $h=1$, we have $\gcd(h,e)=1$ and $\bar rL_h=1$.
Now assume $e\ge3$. Then Lemma~\ref{lem2.3} and Proposition~\ref{prop:quotient-divisibility} yield that $1\le\bar r<L$, $\bar r\equiv1\pmod q$, and $q\mid n_k$ for every $1\le k<e$.  Proposition~\ref{prop:exponent-condition}
therefore gives some $1\le h<e$ with $\gcd(h,e)=1$ such that $\bar rL_h\equiv1\pmod L$.
Thus in either case there exists $1\le h<e$ with $\gcd(h,e)=1$ and $\bar rL_h\equiv1\pmod L$.
Since $r\equiv\bar r\pmod L$, we also have
$rL_h\equiv1\pmod L$. Using
$q^h-1=(q-1)L_h$ and $q^e-1=(q-1)L$, this is equivalent to
$r(q^h-1)\equiv q-1\pmod{q^e-1}$. Hence all three conditions are necessary.

Conversely, suppose that $(-a)^L\ne1$, $\gcd(r,q-1)=1$,
and that for some $1\le h<e$ with $\gcd(h,e)=1$,
$r(q^h-1)\equiv q-1\pmod{q^e-1}$. Equivalently, $rL_h\equiv1\pmod L$.
In particular, $\gcd(r,L)=1$. Together with
$\gcd(r,q-1)=1$, this gives $\gcd(r,q^e-1)=1$,
so $ g(X)=X^r$ induces a permutation of $\mathbb F_{q^e}$.
Moreover, the exponent congruence gives $rq^h\equiv r+q-1\pmod{q^e-1}$. Hence, $F_{r,a}(X)=\left(X^{q^h}+aX\right)\circ X^r$.

It remains to show that $Y^{q^h}+aY$ permutes
$\mathbb F_{q^e}$. Since $\gcd(q^h-1,q^e-1)=q^{\gcd(h,e)}-1=q-1,$
the equation $Y^{q^h-1}=-a$ has a solution in
$\mathbb F_{q^e}^*$ if and only if $(-a)^L=1$.
Thus $(-a)^L\ne1$ implies that $Y^{q^h}+aY$ has no nonzero
root, and hence it is a permutation of $\mathbb F_{q^e}$.
Therefore both $X^r$ and $X^{q^h}+aX$ are permutation
polynomials, and so is $F_{r,a}$.   \qed

\section*{Acknowledgements}

\paragraph{Related concurrent work.}
After completing the present proof, we became aware of a recent preprint by Fan~\cite{Fan2026}, which independently obtains the same classification and proves the Conjecture~\ref{conj}.
The two approaches are substantially different. Fan's proof uses Hermite power sums, Lucas' theorem, Farey intervals, and a digit-forcing argument, whereas our proof is based on a Galois-ring lifting of the Hermite power sums and \(p\)-adic divisibility of the quotient parameters. In particular, our main arithmetic step is to prove \(q\mid n_h\) for all \(1\le h<e\) by excluding bad quotient parameters.


%
\section*{Declarations}
\begin{conflict of interest} {\rm There is no conflict of interest.}
\end{conflict of interest}

%

\end{document}